\documentclass[11pt]{article}
\usepackage[T1]{fontenc}
\usepackage[utf8]{inputenc}
\usepackage[greek,english]{babel}

\oddsidemargin \evensidemargin
\usepackage[ruled,inoutnumbered]{algorithm2e}

\usepackage[style=alphabetic]{biblatex}
\usepackage{enumitem}
\setenumerate{label=\textup{(\roman*)}}

\usepackage{amsmath, mathtools}
\usepackage{amssymb}
\usepackage{amsthm, thmtools}
\usepackage{tikz-cd}
\usepackage{braket}
\usepackage{bm}
\usepackage{xspace}
\usepackage{alphabeta}

\usepackage[frozencache]{minted}
\setminted[julia]{frame=lines,
rulecolor=\color{white!80!black},
fontsize=\small,
numbers=right,
numbersep=-5pt,
obeytabs=true,
encoding = utf8,
tabsize=4}

\usepackage{todonotes}
\usepackage{csquotes}

\usepackage[colorlinks]{hyperref}
\usepackage[nameinlink]{cleveref}
\hypersetup{
  linkcolor=[rgb]{0.3,0.3,0.6},
  citecolor=[rgb]{0.2, 0.6, 0.2},
  urlcolor=[rgb]{0.6, 0.2, 0.2}
}

\declaretheorem[numberwithin=section]{theorem}
\declaretheorem[numberlike=theorem]{lemma, corollary,proposition}
\declaretheorem[numberlike=theorem,style=definition]{definition}
\declaretheorem[numberlike=theorem,style=definition]{example}
\declaretheorem[numbered=no,style=remark]{remark}
\declaretheoremstyle[
notefont=\bfseries, notebraces={(}{)}
]{conjecture}

\newenvironment{examplex}
{\pushQED{\qed}\example}
{\popQED\endexample}

\newcommand{\bfx}{\mathbf{x}}
\newcommand{\bfa}{\mathbf{a}}
\newcommand{\bfu}{\mathbf{u}}
\newcommand{\bfv}{\mathbf{v}}

\newcommand{\calI}{\mathcal{I}}

\newcommand{\numberset}{\mathbb}

\newcommand{\C}{\mathbb{C}}

\newcommand{\Z}{\numberset{Z}}

\newcommand{\R}{\numberset{R}}

\newcommand{\PP}{\mathbb{P}}

\definecolor{cof}{RGB}{219,144,71}
\definecolor{pur}{RGB}{186,146,162}
\definecolor{greeo}{RGB}{91,173,69}
\definecolor{greet}{RGB}{52,111,72}

\DeclareMathOperator*{\argmax}{argmax}
\DeclareMathOperator{\Gr}{Gr}

\title{The SagbiHomotopy.jl package for solving polynomial systems}
\author{Barbara Betti and Viktoriia Borovik}
\date{}
\begin{document}

\maketitle

\begin{abstract}
 We present the \texttt{Julia}  package \texttt{SagbiHomotopy.jl} for solving systems of polynomial equations using numerical homotopy continuation. The package introduces an optimal choice of a~start system based on SAGBI homotopies. For square horizontally parameterized systems, where each equation is a linear combination of a given set of polynomials, SAGBI homotopies significantly reduce the number of solution paths to track compared to polyhedral homotopies currently used by default in most software for numerical homotopy continuation. We illustrate our framework with a variety of examples, including problems arising in chemistry and physics.
\end{abstract}

\section{Introduction}

Solving polynomial systems is a difficult problem that constantly forces the development of new methods. The application of computational numerical techniques is increasing, and one of the main ones is \emph{numerical homotopy continuation}. This method approximates the sets of solutions to polynomial systems defining zero-dimensional varieties. It is based on deforming a system of polynomial equations into another polynomial system with known or easily computable solutions. Once this system is solved, the isolated complex solutions are deformed into solutions to the original system by the same deformation. 

This applies when we have a system $H_{\mathbf{t}}(\mathbf{x})=0$ in variables $\mathbf{x}=(x_1,\dots,x_n)  \in \C^n$ parameterized by polynomials in $\mathbf{t}=(t_1,\dots,t_m) \in \C^m$. Assume that we can compute the solutions for special values~$\mathbf{t}_0$ of the parameters, namely the \emph{start parameters}.  Let $\mathbf{t}_1$ be the \emph{target parameters} defining the system $H_{\mathbf{t}_1}(\mathbf{x})=0$ that we want to solve. We construct a homotopy $H(\mathbf{x},\mathbf{t})$,
\[ H(\mathbf{x},\mathbf{t}_0)=H_{\mathbf{t}_0}(\mathbf{x}), \quad  H(\mathbf{x},\mathbf{t}_1)=H_{\mathbf{t}_1}(\mathbf{x}),\]
such that the domain of the polynomials in $\mathbf{t}$ is path connected. This requirement is fundamental in order to create a path between the start and the target system. If the system $H_{\mathbf{t}}(\mathbf{x})=0$ has the same number of isolated zeros for every $\mathbf{t} \neq \mathbf{t}_0$ and the start system has at least that many isolated zeros, we can track each solution along the continuous solution path using a predictor-corrector methods. The \emph{parameter continuation theorem} by Morgan and Sommese \cite{MS1989} (see also \cite{mePaul}) asserts that solution paths exist. We obtain all isolated solutions for $H_{\mathbf{t}_1}(\mathbf{x})=0$. 
When the start system has exactly the same number of isolated complex zeros as the target system, we say that a homotopy is \emph{optimal}, since we will need to track the minimal number of paths to still get \emph{all}  the solutions of the target system.
This technique has been implemented in several software packages, such as \texttt{Bertini}~\cite{bates2013bertini}, \texttt{PHCpack}~\cite{verschelde1999phcpack} and \texttt{HomotopyContinuation.jl}~\cite{HomotopyContinuation.jl}. Our focus is on the latter, effectively implemented by Paul Breiding and Sascha Timme, since we work in the open-source programming language \texttt{Julia}~\cite{bezanson2012julia}. 

However, the construction of an optimal start system is a difficult problem that has not yet been solved in the general case. Instead, we consider polynomial systems as part of a family of polynomial systems for which an optimal homotopy can be constructed for almost all members of this family. Some well-known good start systems that bound the number of complex solutions of the target system from above are given by \emph{total degree}, \emph{multinomial}, and \emph{polyhedral} homotopies~\cite{polyhedralhomotopy}. 

In this paper we consider the family of square polynomial systems defined on a multiprojective variety $X$ admitting a toric degeneration. 
In practice, we assume that $X$ is parameterized by a~\emph{SAGBI basis} which is a \emph{Khovanskii basis} for certain valuations, see~\cite{KM}. This assumption is satisfied by many varieties often found in applications, such as Grassmannians. We implement the algorithm for constructing a start system in the registered \texttt{Julia} package \texttt{SagbiHomotopy.jl}. The source code is available at the online GitHub repository \href{https://github.com/Barbarabetti/SagbiHomotopy.jl}{SagbiHomotopy} and the examples discussed in the paper can be found and reproduced at the MathRepo (\cite{mathrepo}) page:
\begin{equation}\label{eq: mathrepo}
\hbox{\url{https://mathrepo.mis.mpg.de/SagbiHomotopy}}
\end{equation}
In \cite{NumericalHomotopies} the authors present a numerical homotopy continuation algorithm for solving systems of
polynomial equations on a variety $X$ in the presence of a finite {Khovanskii} basis. They consider varieties embedded in a single projective space. 
Theoretically, we use the same algorithm as they do, but extend it to the case of a multiprojective variety $X$ and develop an alternative implementation that takes advantage of having a finite SAGBI basis as a parameterization of $X$. Moreover, we avoid computing the defining ideal of the variety $X$ and instead use the degeneration of the coordinate ring $\C[X]$ into a monomial algebra. In order to distinguish our procedures, we consistently use the term ``SAGBI basis'' instead of the more general ``Khovanskii basis''.

The algorithm from~\cite{NumericalHomotopies} has not been implemented in the general setting because of two main challenges. First, the polynomials providing local coordinates on the variety $X$ cannot, in general, be derived directly from the given polynomial system. Second, a finite Khovanskii basis for a given variety 
$X$ may not exist and, when it does, it depends on the choice of a valuation on $\C[X]$. The first issue remains the major obstacle, and currently the user must provide the system in a specific form, as described in \Cref{sec:4}. However, the second issue can be addressed in the SAGBI bases setting. Namely, given a finite set of polynomials, we can determine whether there exists a valuation – in this case a term order – for which they form a SAGBI basis, or we can reveal that no such term order exists. This is implemented using the \emph{SAGBI detection} algorithm \cite{sagbidetection}. 


Our main contribution is the implementation of the SAGBI homotopy algorithm in the general setting. Our algorithm is implemented in \texttt{Julia} (\texttt{v1.10.5}) using \texttt{Oscar.jl} (\texttt{v1.2.2}) and \texttt{HomotopyContinuation.jl} (\texttt{v2.13.0}). The paper is organized as follows. In \Cref{sec:2} we recall the basics of SAGBI bases and toric degenerations. In \Cref{sec:3} we present our numerical homotopy algorithm, which uses SAGBI bases for a homogeneous coordinate ring of a variety. In \Cref{sec:4} we introduce the package \texttt{SagbiHomotopy.jl} and demonstrate the use of its main functions with examples. Finally, in \Cref{sec:5} we apply our algorithm to parameterized families of polynomial systems arising in the sciences. 



\section{SAGBI bases and toric degenerations}\label{sec:2}

Our main tool for constructing a SAGBI homotopy is a \emph{Gröbner toric degeneration} of a projective or a multiprojective unirational variety. Algebraically, for a given unirational variety, this degeneration is described by a SAGBI basis for its homogeneous (respectively multihomogeneous) coordinate ring.

\bigskip

In this section, $X$ is either a projective variety in $\PP^k$ or a multiprojective variety in
\[\mathbf{P}^{k_1, \dots, k_m} := \PP^{k_1} \times \dots \times \PP^{k_m}.\]
\begin{definition}
A \emph{toric degeneration} of $X$ is a family of varieties~$\mathcal{X}$ together with a \emph{flat morphism} ${\pi\colon \mathcal{X} \to  \C^1}$ such that  each fibre over $t \in \C\setminus \{0\}$ is isomorphic to $X$ and the fibre over $t=0$ is a
toric variety $X_0$.  These varieties are respectively called \emph{general fibre} and \emph{special fibre}.
\end{definition}    
The advantage of the flat family is that we can determine many properties of a general fibre $X$ by studying the special fibre $X_0$. Indeed, varieties of the flat family $\mathcal{X}$ share common properties such as dimension, degree, Hilbert function, Cohen-Macaulayness, normality, etc. For more details we refer to \cite{grothendieck1966eléments, bruns2022determinants}. 
In particular, we will heavily rely on the  following widely known lemmas.
\vspace{-0.5cm}
\begin{lemma} \label{lem: kush}
    Let $\mathcal{X} \to  \C^1$ be a flat family with fibres $X_t\subseteq\PP^k$ for $t\in \C^1$, and let $L\subseteq \PP^k$ be a~general linear subspace of codimension equal to ${\rm dim}(X_t)$. Then the linear section $X_t\cap L$ consists of the same number of points for every fibre, namely $\deg(X_t)$ many points, and this is independent of the choice of $t\in \C^1$.
\end{lemma}
\noindent The same statement is true in the multiprojective setting. 
\begin{lemma} \label{lem: bern}
    Let $\mathcal{X} \to  \C^1$ be a flat family with fibres $X_t\subseteq\mathbf{P}^{k_1, \dots, k_m}$ for $t\in \C^1$, and let 
    \[
\mathbf{L}:= L_1 \times \cdots \times L_m \subset \mathbf{P}^{k_1, \ldots , k_m},
\]
be a product of general linear subspaces $L_r \subseteq \PP^{k_r}$ of codimension~$i_r$ for $1 \leq r \leq m$ such that $i_1 + \cdots + i_m = \dim X_t$. Then, for any such tuple $(i_1, \dots, i_m)$, the linear section $X_t\cap \mathbf{L}$ consists of the same number of points for every fibre, namely $\mathbf{e}(i_1, \dots, i_m)(X_t)$ many points, and this is independent of the choice of $t\in \C^1$.
\end{lemma}
\noindent The functions $\mathbf{e}(i_1, \dots, i_m)(X)$ are called \emph{multidegrees} of the multiprojective variety $X \subseteq \mathbf{P}^{k_1, \ldots, k_m}$. They correspond to coefficients of a homogeneous polynomial of degree ${\rm codim}(X)$ \cite[Ch.8]{73608}.

\smallskip

When a flat family is a toric degeneration, the discussed properties are easy to understand, since toric varieties are well-studied and have a nice combinatorial description. In particular, the degree of a projective toric variety is given by the normalized volume of the corresponding polytope in $\R^n$. To simplify the statements, we will assume that all the polytopes associated to toric varieties are full-dimensional, that is, we assume $\dim X = n$. The following result is due to Kushnirenko \cite{Kouchnirenko1976}.

\begin{theorem}[Kushnirenko \cite{Kouchnirenko1976}] \label{thm: Kush}
Let $X\subseteq \PP^k$ be a projective toric variety parameterized by Laurent monomials $\bfx^{\bfa_0},\dots, \bfx^{\bfa_k}$, with $\bfa_i\in \Z^n$, and let $P\subseteq \R^n$ be the convex hull of $\{\bfa_0,\dots, \bfa_k\}$. Then the degree of $X$ is the normalized volume of $P$:
\[ \deg(X)=n!\ {\rm Vol}(P).\]
In particular, ${\rm Vol}(P)$ is the leading coefficient of the Hilbert polynomial of $X$.
\end{theorem}
 Similarly, multidegrees of a multiprojective toric variety are given by the mixed volume of the corresponding polytopes. The result below is also known as the BKK (Bernstein-Khovanskii-Kushnirenko) Theorem for \emph{sparse} polynomial systems, see \cite{bernstein}.
\begin{theorem}[{\cite[Theorem 5.4]{KavehKhov2012}}] \label{thm: MixVol}
    Let $X\subseteq \mathbf{P}^{k_1, \ldots, k_m}$ be a multiprojective toric variety parameterized by Laurent monomials $(\bfx^{\bfa_{1j}})_{j=0,\dots, k_1}, \dots, (\bfx^{\bfa_{mj}})_{j=0,\dots,k_m}$ with $\bfa_{ij}\in \Z^n$. Let $P_i$ be the convex hull of $\{\bfa_{ij}: j=0,\dots, k_i\}$. Then the multidegree of $X$ is the mixed volume 
    \[\mathbf{e}(i_1, \dots, i_m)(X)= \mathrm{MV}(\underbrace{P_1, \dots, P_1}_{i_1}, \dots, \underbrace{P_m, \dots, P_m}_{i_m}), \quad i_1 + \dots + i_m = n.\]
    
\end{theorem}
\noindent Clearly, since ${\rm MV}(P,\dots,P)$ is equal to $n!\ {\rm Vol}(P)$, \Cref{thm: Kush} is a direct corollary of \Cref{thm: MixVol}. 

\bigskip

We~focus on the toric degenerations which come from \emph{Gröbner degenerations}. More precisely, let~$\C[X]$ be the coordinate ring of a variety $X$, that can be represented as a factor ring $\C[\mathbf{z}]/I$, where~$\mathbf{z}=z_0,\dots,z_k$, and $I=\calI(X)$ is the defining ideal of $X$. We consider  its initial ideal:
$$\mathrm{in}_{v}(I)  = \mathrm{span}_{\C}\{\mathrm{in}_{v}(f) \mid f\in I\}$$ 
with respect to a weight vector $v\in \Z^{k+1}$ and the corresponding variety $Y=\mathcal{V}(\mathrm{in}_{v}(I))$ it defines. We refer to \cite[Ch.1]{sturmfelsgröbner} for an introduction to weight representation of term orders.

\medskip

We extend our polynomial ring to $\C[\mathbf{z},t]$ by adding an extra variable $t$. The torus~$\mathbb{C}^\times$ acts on~$\C[\mathbf{z}]\subseteq \C[\mathbf{z},t]$ by~$t \cdot \mathbf{z}^\alpha := t^{-v \cdot \alpha} \mathbf{z}^\alpha$. For a polynomial~$f = \sum c_\alpha \mathbf{z}^{\alpha} \in \C[\mathbf{z}]$ we set~$v(f) := \underset{\alpha}{\max} \{ v \cdot \alpha\}$ and define the~\emph{homogenization} of the polynomial $f$ and the ideal~$I$  with respect to a weight $v$  as:
\[\label{weight_degeneration}
f_t^{v} := t^{v(f)} (t \cdot f), \quad I_t^{v} := (  f_t^{v} \mid f\in I) \subseteq \C[\mathbf{z},t].  \]
The following flat family of varieties is called a \emph{Gröbner degeneration} with respect to a weight $v$.
\begin{proposition} [{\cite[Theorem~15.17]{Eisenbud}}] \label{prop: flatfactorrings}
    The $\C[t]$-algebra $\C[\mathbf{z}, t]/I_t^v$ is flat as a $\C[t]$-module. The special fibre over $t=0$ is $\C[\mathbf{z}]/\mathrm{in}_{v}(I)$ and the general fibre is isomorphic to $\C[\mathbf{z}]/I$.
\end{proposition}

\begin{remark}
When $\mathrm{in}_{v}(I)$ is a binomial prime ideal, i.e., a toric ideal, the special fibre is a coordinate ring of a toric variety and the introduced Gröbner degeneration gives us a toric degeneration. On the other hand, since we are considering embedded varieties, every toric degeneration arising from a SAGBI (or even Khovanskii) basis is indeed a Gröbner degeneration. We refer to Section 3.3 of the survey \cite{bossinger2023surveytoricdegenerationsprojective} for a detailed discussion.
\end{remark}

\bigskip

Now, assume that the~variety $X$ is unirational and parameterized by a finite set of polynomials $$\mathcal{B}=\{b_0,\dots,b_k\}\subseteq \C[\mathbf{x}]=\C[x_1,\dots,x_n].$$ Equivalently,
the coordinate ring of $X$ is a subalgebra of the polynomial ring $\C[\mathbf{x}]$ generated by $\mathcal{B}$:
\[ S:= \C[b_0,\dots,b_k], \quad S \simeq \C[\mathbf{z}]/I. \]
Consider a term order~$\succ$ on~$\C[\mathbf{x}]$. Analogously to the ideals case, the initial algebra~$\mathrm{in}_{\succ}(S)$ of~$S$ is the monomial subalgebra in~$\C[\mathbf{x}]$ spanned by the~$\succ$-initial terms of the polynomials from $S$:
$$ \mathrm{in}_{\succ}(S) = \mathrm{span}_{\C}\{\mathrm{in}_{\succ}(f) \mid f \in S\}.$$
\vspace{-0.8cm}
\begin{definition}
The finite set of polynomials $\mathcal{B}$ is \emph{a SAGBI basis} for~$S$ if
$ \mathrm{in}_{\succ}(S) = \C[\mathrm{in}_{\succ} (b) \mid b \in \mathcal{B}]$.
\end{definition}
It follows immediately from the definition that an algebra $S$ has a finite SAGBI basis with respect to $\succ$ if and only if the initial algebra~$\mathrm{in}_{\succ}(S)$ is finitely generated. SAGBI bases were introduced by Robbiano and Sweedler in \cite{robbiano2006subalgebra} as an analog of Gröbner bases for ideals for subalgebras~of~polynomial rings. The crucial difference is that not every polynomial subalgebra admits a finite SAGBI basis. The reason is that noetherianity of the polynomial ring is not inherited by subalgebras. Thus, an initial algebra of a finitely generated subalgebra may not necessarily be finitely generated itself.

\begin{remark}[Khovanskii basis] 
    A generalization of SAGBI bases is defined in \cite{KM} for any finitely generated algebra, not necessarily contained in a polynomial ring, equipped with a valuation to an ordered abelian group. A set $\mathcal{B}\subseteq A$  is a \emph{Khovanskii basis} for a finitely generated algebra equipped with a valuation $(A,\nu)$ if its image generates the associated graded algebra ${\rm gr}_\nu(A)$.    
    SAGBI bases for subalgebras of polynomial rings fall into this broader class because they arise as Khovanskii bases associated to specific valuations, called \emph{leading term valuations}, induced by term orders:
    \[ \nu: A\setminus\{0\}\rightarrow (\Z^n,\succ), \quad f\mapsto \text{ exponent of } \mathrm{in}_{\succ}(f).\]
    
\end{remark}

We show that the existence of a finite SAGBI basis for $S=\C[X]$ implies the existence of a toric Gröbner degeneration of~$X$. Let~$\omega \in \mathbb{Z}^{n}$ be any weight vector compatible with~$\succ$ on~$\mathcal{B}$. That is, for~$b_i \in \mathcal{B}$ with~$\mathrm{in}_{\succ}(b_i) = c \, \mathbf{x}^{\alpha_i}$, for any nonzero term~$c \, \mathbf{x}^\beta$ of~$b_i$, we have~$\omega \cdot \alpha_i \geq \omega \cdot \beta$ with equality only if~$\alpha_i = \beta$. We set $\mathcal{A}$ to be the $n \times (k+1)$ matrix whose columns are the leading exponents $\alpha_i$:
\[\mathcal{A} = \left[
\begin{array}{c|c|c|c}
\alpha_0 & \alpha_1 & \cdots & \alpha_k
\end{array}
\right],\]
and consider the weight vector
$v := \mathcal{A}^{T}\omega$ in $\R^{k+1}$. 
\begin{theorem} [{\cite[Theorem 11.4]{sturmfelsgröbner}}]
 The set $\mathcal{B} = \{b_0,\dots,b_k\}$ is a SAGBI basis for the subalgebra~$S \simeq \C[\mathbf{z}]/I$ if and only if the initial ideal $\mathrm{in}_{v}(I)$ is a toric ideal defining the variety $\mathrm{Spec}(\mathrm{in}_{\omega}(S))$.
\end{theorem}
As before, we consider the homogenization of the polynomials from $\mathcal{B}$ and the subalgebra $S$:
\[\label{weight_degeneration_algebra}
b_t^{\omega} := t^{\omega(b)} (t \cdot b), \quad S_t^{\omega} := \C[b_t^{\omega} \mid b \in \mathcal{B}]. 
\]
The following results give rise to our algorithm of the SAGBI homotopy, inspired by the Gröbner homotopy constructed in \cite[Section 2.2]{NumericalHomotopies}. The proof is presented in  \Cref{appendix}.
\begin{proposition} \label{prop: DeformedAlgebrasIsom}
    Let $\mathcal{B}  = \{b_0,\dots,b_k\}$ form a SAGBI basis with respect to~$\omega$ and take $v = \mathcal{A}^{T}\omega$.  Then for every $t \in \C^1$ the subalgebra $S_t^{\omega}$ is isomorphic to  $\C[\mathbf{z},t]/I_t^{v}$.
\end{proposition}
As a consequence from \Cref{prop: flatfactorrings} and \Cref{prop: DeformedAlgebrasIsom} we obtain the following result.
\begin{corollary}
The family of subalgebras~$S_t^{\omega}$ is flat over~$\mathbb{C}[t]$ with~$S_1^\omega = S$  and~$S_0^\omega = \mathrm{in}_{\omega}(S)$. The special fibre~$\mathrm{in}_{\omega}(S)$  is the coordinate ring of a toric variety.    
\end{corollary} 
\begin{remark}
In this section we assume that the set $\mathcal{B}$ consists of arbitrary polynomials, since the results apply to any affine variety parametrized by $\mathcal{B}$. In what follows, we focus on projective and multiprojective varieties, where as parameterization we consider the parameterization of the corresponding affine cone. For example, in the projective case, the set $\mathcal{B}$ can be thought of as consisting of polynomials of the same degree.

To represent the rational map into projective space, we will use an alternative notation involving an auxiliary variable $s$, representing the map by a tuple $\{s b_0, \dots, s b_k\}$. Similarly, the rational map into the product of projective spaces is represented by
\[
\mathbf{s} \mathcal{B} \coloneqq \bigcup_{r=1}^m s_r \mathcal{B}_r,
\]
where $\mathbf{s} = (s_1, \dots, s_m)$ are auxiliary variables and each $s_r \mathcal{B}_r = \{s_r b_{r,0}, \dots, s_r b_{r,k_r}\}$ corresponds to a component of the product. This set $\mathbf{s} \mathcal{B}$ generates a multigraded polynomial subalgebra, which serves as a general fibre in the toric degeneration we construct.
\end{remark}
\section{SAGBI homotopy}\label{sec:3}

   Let us consider a unirational variety $X$ in a multiprojective space parameterized by $\mathcal B\coloneqq \cup_{r=1}^m \mathcal{B}_r$, where $\mathcal{B}_r=\{ b_{r,0},\ldots ,b_{r,k_r}\}\subseteq \C[\mathbf x]$. More precisely, $X$ is the image closure of the rational map
    \begin{align*}
    \phi\colon \C^n &\dashrightarrow  \mathbf{P}^{k_1, \dots, k_m} = 
    \PP^{k_1}\times\dots\times \PP^{k_m},\\ 
    \mathbf{x}&\mapsto
    \big(\, [b_{1,0}(\mathbf x) : \ldots : b_{1,k_1}(\mathbf x)],\ \dots\ ,\ [b_{m,0}(\mathbf x) : \ldots : b_{m,k_m}(\mathbf x)]\, \big).
\end{align*} 
\vspace{-0.6cm}
\begin{definition} \label{def: base locus}
    The  \emph{base locus} of~$X$ is the following algebraic set:
    \[ \mathbb{B}_X \coloneqq \bigcup_{r=1}^m \mathcal{V}(\mathcal{B}_r).\] 
\end{definition}
\begin{remark}
The subvariety  on which the map $\phi$ is not defined might be strictly contained in  the base locus $\mathbb{B}_X$, as introduced above.
It is obtained by intersecting the base loci for all possible representatives of $\phi$. For our purposes, it is necessary to consider $\mathbb{B}_X$.
\end{remark}
    The multigraded coordinate ring of $X$ is generated by polynomials from $\mathbf{s}\mathcal{B}\coloneqq\cup_{r=1}^m s_r\mathcal{B}_r$, where $\mathbf s = (s_1,\dots,s_m)$ are auxiliary variables  and $s_r\mathcal{B}_r=\{ s_rb_{r,0},\dots,s_rb_{r,k_r}\}$:
    \[ \C[X] := \C[s_r\cdot b_{r,j}(\mathbf x): 1\leq r\leq     m,\ 0\leq j\leq k_r] \subseteq \C[\mathbf s , \mathbf x].\]
    The multigrading is given by considering the degree with respect to $s_1,\dots,s_m$: $\deg(s_i)=e_i\in \Z^m$.

\bigskip
 
Suppose that the set of polynomials $\mathbf{s}\mathcal{B}$ forms a SAGBI basis for the coordinate ring of $X$ with respect to a term order $\succ$. We extend it to a degree compatible term order on $\C[\mathbf s,\mathbf{x}]$ and represent it by a weight vector $\widehat\omega= (-1,\dots,-1,\, \omega)\in \Z^{n+m}$, assigning weight $-1$ to $s_1,\dots,s_m$ and $\omega_i$ to each variable $x_i$. From the previous section we have that~$X$ admits a flat degeneration to a toric variety~$X_0$ whose coordinate ring is $\C[X_0]= \mathrm{in}_{\widehat\omega}(\C[X])$. 
    Explicitly, $X_0$ is  the image closure  of the following monomial map:
\begin{align*}
    \phi_0 \colon \C^n &\dashrightarrow \PP^{k_1}\times\dots\times \PP^{k_m},\\ 
    \mathbf{x} &\mapsto
    \big(\,[\mathrm{in}_{\omega}(b_{1,0}(\mathbf x)):\ldots:\mathrm{in}_{\omega}(b_{1,k_1}(\mathbf x))],\ \ldots\ ,\ [\mathrm{in}_{\omega}(b_{m,0}(\mathbf x)):\ldots:\mathrm{in}_{\omega}(b_{m,k_m}(\mathbf x))]\,\big).
\end{align*} 
    We consider a square system $ \mathcal{F}:= \{ f_1=\cdots= f_n=0\}$ with equations given by linear forms in the parameterization $\mathcal{B}$. In particular, each equation is a linear combination of a specific set $\mathcal B_r$:
 \begin{equation}\label{linearcombination}
    f_i = \sum\limits_{j=0}^{k_r} c_{ij}\, b_{r,j},  \quad \text{for some } 1\leq r\leq m. 
\end{equation}
    Systems of the form~\eqref{linearcombination} are called \emph{horizontally parameterized}. When all the polynomials from $\mathcal{B}$ are monomials, the system~\eqref{linearcombination} becomes sparse. For more details on sparse systems we refer to \cite{polyhedralhomotopy}. Analogously to the case of sparse systems, we say that the system $\mathcal F$ is \textit{unmixed} when~$m=1$, \textit{fully mixed} when~$m=n$, and \textit{semimixed of type $(k_1,\dots,k_m)$} for arbitrary~$m$.

 \begin{definition}\label{sagbihomotopyDef}
     The \emph{SAGBI homotopy} is the system $\mathcal{F}_t^\omega$ obtained by replacing every $b_{r,j}$ in \eqref{linearcombination} with its homogenization  $(b_{r,j})_t^\omega$ with respect to the weight vector $\omega$:
        \begin{equation} \label{eq: sagbi system}
            \mathcal{F}_t^\omega= \{ f_{1,t}^\omega=\cdots= f_{n,t}^\omega=0\}, \quad f_{i,t}^\omega= \sum\limits_{j=0}^k c_{ij}(b_{r,j})_t^\omega =0.
        \end{equation}
 \end{definition}
     We observe that we recover $\mathcal{F}=\mathcal{F}_1^\omega$ from the SAGBI homotopy by setting the parameter $t=1$. For $t=0$, we get a semimixed sparse system $\mathcal{F}_0^\omega$ defined on the toric variety~$X_0$ via linear equations in the initial terms of $b_{r,j}$:
     \begin{equation} \label{eq: sparsesys}
            \mathcal{F}_0^\omega= \{ f_{1,0}^\omega=\cdots= f_{n,0}^\omega=0\}, \quad f_{i,0}^\omega= \sum\limits_{j=0}^k c_{ij} \cdot \mathrm{in}_{\succ}(b_{r,j}) = 0.
        \end{equation}

\begin{examplex} \label{ex: homogenization sagbi homotopy system}
    We consider a horizontally parameterized system $\mathcal{F}$ given  by $\mathcal{B} = \mathcal{B}_1 \cup \mathcal{B}_2$, where
    \[\mathcal{B}_1=\{x,y, x^2 + y^2,1 \}, \quad \mathcal{B}_2=\{y, z, x^2 + y^2, x^3 + z^3 \}\]
    and $\mathcal{F} =\{f_1=f_2=0\}$ consists of linear combinations of polynomials from $\mathcal{B}_1$ and $\mathcal{B}_2$, respectively.
    The set $s_1\mathcal{B}_1\cup s_2\mathcal{B}_2$ is a SAGBI basis with respect to the weight $\widehat\omega = (-1,-1,-2,-3)$. 
    We compute the homogenization of $\mathcal{B}_1$ and $\mathcal{B}_2$ with respect to the weight $\omega = (-2,-3)$:
\[ (\mathcal{B}_1)_t^\omega = \{ x, \, y,\, x^2 + t^2 y^2,\,  1 \}, \quad (\mathcal{B}_2)_t^\omega = \{ y,\, z,\, x^2 + t^2y^2,\, x^3 + t^6 z^3 \}.\]
  The functions that perform this computation are presented in \Cref{ex: showing functions}.  For instance, the equation corresponding to $f_2$ in the SAGBI homotopy is computed as:
    \[ 
    f_{2,t}^\omega =  c_{20}\cdot y + c_{21}\cdot z + c_{22}\cdot (x^2 + t^2y^2) + c_{23}\cdot  (x^3 + t^6z^3). \qedhere \]
\end{examplex}
The solutions of the sparse system $\mathcal{F}_0^\omega$ from \eqref{eq: sparsesys} can be optimally computed via the polyhedral homotopy, as constructed in~\cite{polyhedralhomotopy}. Thus, we obtain a family of systems parameterized by $t\in \C$ that we are able to solve for $t=0$, which we consider as a start parameter in the homotopy continuation method. To deform the start parameter to the target parameter $t=1$, we consider a continuous random path $h\colon [0,1]\rightarrow \C$ such that $h(0)=0$, $h(1)=1$ and we consider the homotopy given by
 $H(\mathbf{x},t) := \mathcal{F}_{h(t)}^\omega$.
By tracking each solution of the start system $H(\mathbf{x},0) = \mathcal{F}_0^{\omega}$ along the solution path we obtain all solutions to the target system $H(\mathbf{x},1) = \mathcal{F}$. 

\medskip

We summarize the previous discussion in \Cref{alg: two-step sagbi algorithm}, that is the main result of the paper.  The first step is conducted via the SAGBI detection algorithm \cite[Algorithm 2]{sagbidetection}.
\bigskip

\begin{algorithm}[H]
\caption{Two-step SAGBI Homotopy Algorithm} \label{alg: two-step sagbi algorithm}
\SetAlgoLined
\KwIn{Linear equations~$\mathcal{F} = \{f_1,\ldots,f_n\}$ of the form \eqref{linearcombination} and finite sets $\mathcal{B}_r, \; 1\leq r \leq m$.}
\KwOut{Isolated complex solutions to the system~$\mathcal{F}$.}
1. Detect a weight $\omega$, for which polynomials in $\cup_{r=1}^m s_r\mathcal{B}_r$ form a SAGBI basis\;

2. Compute~$\mathcal{F}_t^{\omega}$  by substituting~$(b_{r,j})_t^{\omega}$ to $b_{r,j}$ in the expression for~$f_i$\;

3. Solve the sparse system~$\mathcal{F}_0$ via the polyhedral homotopy \cite{polyhedralhomotopy}\;

4. Track these solutions to the solutions of~$\mathcal{F}$ at~$t = 1$.
\end{algorithm}

\subsection{Optimality}
 We assume that the SAGBI homotopy is a general square system, i.e., the number of equations is equal to the number of variables~$n$ and the coefficients~$c_{ij}$ of the system are general. Since the number of solutions depends on the degree of the parametrization map $\phi$, we also suppose that this is finite and equal to the degree of the monomial map $\phi_0$. We recall the definition for completeness.
 \vspace{-0.5cm}
\begin{definition}
    Let $X$ be the  image closure of the rational map $\phi:\C^n\dashrightarrow  \mathbf{P}^{k_1, \dots, k_m}$. 
    If there exists an open subset $U\subseteq X$ such that the fibre $\phi^{-1}(x)$ consists of exactly $d$ points for every $x\in U$, we say that the map $\phi$ is \emph{finite of degree} $d$:
    \[\deg(\phi)=\mid{\phi^{-1}(x)}\mid=d \text{ for generic } x\in X.\]
\end{definition}
\begin{remark}
    The rational map between two irreducible varieties $\phi: V \to W$ is generically finite if and only if the field extension $[\mathbb{C}(V):\mathbb{C}(W)]$ of their function fields is finite. The degree $\deg \phi$ is then the degree of this field extension. 
\end{remark}
The following is the key result to show that \Cref{alg: two-step sagbi algorithm} is correct. It follows from the flat degeneration properties, explicitly from \Cref{lem: kush} and \Cref{lem: bern}. SAGBI homotopy algorithm is implied by so called \emph{generalized BKK Theorem} \cite[Theorem 2.5]{oscillators} in the same way the polyhedral homotopy algorithm follows from the usual BKK Theorem, as constructed in \cite{polyhedralhomotopy}.
\begin{proposition} \label{prop: number sols Sagbi homotopy}
     If $\deg(\phi)=\deg(\phi_0)$, then the SAGBI homotopy system has the same number of solutions outside of base locus for every $t\in \C$, that is equal to  $\deg(\phi) \cdot {\rm deg}(X)=\deg(\phi_0)\cdot {\rm deg}(X_0)$. 
\end{proposition}
\begin{corollary}
    If $\deg(\phi)=\deg(\phi_0)$, then the SAGBI homotopy is optimal for the family of general square horizontally parameterized systems of the form~\eqref{linearcombination}.
\end{corollary}
    The assumption $\deg(\phi)=\deg(\phi_0)$ is necessary, since both the cases $\deg(\phi)<\deg(\phi_0)$ and $\deg(\phi)>\deg(\phi_0)$ may occur as showed in the examples below. In the first case the SAGBI homotopy is not optimal, but still finds all the solutions. In the second case, SAGBI homotopy will miss exactly ${\rm deg}(X)(\deg(\phi)-\deg(\phi_0))$ many solutions.

\medskip

    In the examples below, we continue to use parameterizing sets given by SAGBI bases for illustrative purposes, in order to highlight the  challenges for our algorithm. In general, the degrees of the maps $\phi$ and $\phi_0$ do not depend on the presence of SAGBI basis. However, we conjecture that the equality $\deg(\phi)=\deg(\phi_0)$ holds whenever the rational map $\phi$ can be represented by a homogeneous Gröbner basis with respect to some term order.
    
\begin{examplex}
We consider a parameterizing set in three variables $\mathcal{B} = \{x^3,\,y^3,\,x^6y^3+xz^2+y^3\}$. The set $s\mathcal{B}$ is a SAGBI basis with respect to $\widehat\omega = (-1,\, \omega)$, where $\omega=(-3,-3,-8)$. However, the degree of the parameterization induced by $\mathcal{B}$ is $\deg(\phi)=3$ and the corresponding monomial map has degree $\deg(\phi_0)=18$. 
\end{examplex}
The next example appears in the study of chemical reaction networks \cite[Example 5.3]{OBATAKE2024102672}.
\begin{examplex} \label{ex: base locus} 
    The map $\phi$ defining the unirational variety $X\subseteq \PP^3$ of degree 2 parameterized by  $$\mathcal B=\{x(x^2 + y^2 - 2x),\  x(5-4y), \,y(x^2 + y^2 - 2x), \,y(5-4y)\}$$ 
    has degree $2$. This parameterization induces a SAGBI basis $s\mathcal{B}$
with respect to $\widehat\omega=(-1,-2,-1)$. The monomial map has degree $1=\deg(\phi_0)<\deg(\phi) =  2$. The SAGBI homotopy algorithm returns $1\cdot 2=2 $ solutions.
Moreover, the map $\phi$ has a non-empty zero-dimensional base locus. Indeed, there are three common solutions of the polynomials from $\mathcal{B}$: 
    \[ \mathbb B_X = \{ (0,0), \,( 1 - \dfrac{3}{4}i, \dfrac{5}{4}), \,(1+ \dfrac{3}{4}i, \dfrac{5}{4}) \}, \]
which are also missed by our homotopy. In total, a general linear system in coordinates on $X$ has $7$ solutions with four solutions appearing as pull-backs of the points on $X$ and three solutions listed in the base locus above. All the computations will be given in \Cref{ex: base locus code}.
\end{examplex}
\begin{remark}
    Only zero-dimensional components of the base locus lead to missing isolated solutions. Positive-dimensional solution sets are  not tracked by the standard  \texttt{HomotopyContinuation.solve} function either. Notice that, since the polyhedral homotopy computes solutions in the torus, $(0,0)$ will be not in the output of \texttt{HomotopyContinuation.solve} function as well.
\end{remark}    
\subsection{Optimization}
    We recall that the polyhedral homotopy \cite{polyhedralhomotopy} constructs algebraic deformations from \emph{mixed subdivisions} of a polytope, which is defined by the leading terms $\mathrm{in}_{\omega}(b_{r, j})$ from $\cup_{r=1}^m \mathrm{in}_{\omega}(\mathcal{B}_r)$. For each \emph{cell} of this subdivision one extracts a weight $\gamma$. Then the polyhedral homotopy replaces $\mathcal{F}_0$ with its homogenization $\mathcal{F}_{0, t}^{\gamma}$ with respect to $\gamma$. For $t=0$, the system $\mathcal{F}_{0, 0}^{\gamma}$ can be solved with the \emph{Smith Normal Form procedure}. We can define a homotopy directly from the system $\mathcal{F}$ to $\mathcal{F}_{0,0}$. More precisely, for every equation $f_i = \sum c_{ij} b_{r,j}$ with $1 \leq i \leq n$ we have
\begin{equation} \label{one-stepHomotopy}
    (f_i)_t^{\omega, \gamma} := t^{\gamma \cdot \alpha_r}\sum_{j} c_{ij} t^{-\gamma \cdot \alpha_{rj}}\left(\mathbf{x}^{\alpha_{rj}} + \sum_{q} t^{\omega \cdot \alpha_{rj} - \omega \cdot \beta_{jq}} \mathbf{x}^{\beta_{jq}}\right), \quad \alpha_r:= \argmax_{j} \langle \gamma, \alpha_{rj}\rangle.
\end{equation}
Here $\mathbf{x}^{\alpha_{rj}} = \mathrm{in}_{\omega}(b_{r,j})$ and $\left(\mathbf{x}^{\alpha_{rj}} + \sum_{q} t^{\omega \cdot \alpha_{rj} - \omega \cdot \beta_{jq}} \mathbf{x}^{\beta_{jq}}\right) = (b_{r,j})_t^{\omega}$. With this simplification, the SAGBI homotopy algorithm takes the following form. We refer to it as the \emph{one-step} algorithm, as the homotopy is constructed only once.
\bigskip

 \begin{algorithm}[H] \label{alg: One-step SAGBI Homotopy Algorithm}
\caption{One-step SAGBI Homotopy Algorithm}
\SetAlgoLined
\KwIn{Linear equations~$\mathcal{F} = \{f_1,\ldots,f_n\}$ of the form \eqref{linearcombination} and finite sets $\mathcal{B}_r, \; 1\leq r \leq m$.}
\KwOut{Isolated complex solutions to the system~$\mathcal{F}$.}
1. Detect a weight $\omega$, for which polynomials from $\cup_{r=1}^m s_r\mathcal{B}_r$ form a SAGBI basis\;

2. Construct a  mixed subdivision $\mathcal{S}$ of $(\mathrm{in}_{\omega}(\mathcal{B}_1), \ldots, \mathrm{in}_{\omega}(\mathcal{B}_m))$ as in \cite{polyhedralhomotopy}\;

3. For each cell $C_{\gamma}$ of $\mathcal{S}$ compute $\mathcal{F}_{t}^{\omega, \gamma}$ applying \eqref{one-stepHomotopy} for each~$f_i$\;

4. Solve the system~$\mathcal{F}_{0,0} = \mathcal{F}_0^{\omega, \gamma}$ using the Smith Normal Form procedure\;

5. Track these solutions to the solutions of~$\mathcal{F}$ at~$t = 1$.
\end{algorithm}
\section{Functionality}\label{sec:4}


We illustrate how to use the functions of our package with a series of examples. The main function \texttt{sagbi\char`_homotopy} takes as input a linear system whose equations are separated into $m$ blocks that do not share variables: 
\[F = \{f_1 = \cdots = f_n = 0\}, \quad f_i = \sum_{j=0}^{k_r} c_{ij} z_{r,j},  \quad \text{for some } 1\leq r\leq m \]
and a list of polynomials $(b_{r,j})$ in $n$ variables for $1\leq j \leq k_r$ and $1\leq r \leq m$, where $m$ and $k_1,\ldots, k_m$ are fixed.
Optionally it takes also a weight $\omega \in \Z^n$, otherwise such a weight will be computed applying the SAGBI detection algorithm \cite[Algorithm 2]{sagbidetection}. 
\begin{remark}
    In the original implementation of \cite{sagbidetection}, the SAGBI detection algorithm returns all weight vectors for which a given set of polynomials forms a SAGBI basis. To improve efficiency, our implementation stops the search as soon as the first valid weight is identified. This behavior is implemented in the \texttt{SagbiDetection.jl} file of our package. 

Furthermore, since we work exclusively with homogeneous algebras, the SAGBI criterion is greatly simplified. The function \texttt{weightVectorsRealizingSAGBI}, named after its counterpart in \cite{sagbidetection}, computes the weight vector $\widehat\omega$. The \texttt{get\char`_weight} function then outputs the corresponding weight $\omega$, represented in \Cref{sec:3}.
\end{remark}
The output of \texttt{sagbi\char`_homotopy} consists of computed isolated zeros of the polynomial horizontally parameterized square system $\mathcal{F}$, obtained by substituting each $b_{r,j}$ to the corresponding~$z_{r,j}$. 
We illustrate the main function on both unmixed and semimixed cases.
\begin{examplex}[
$n=3,\, m=1,\, k_1=3$] \label{ex: one proj space, degree equal}
    Consider the projective space $\PP^3$ as the image closure of
    \[ \phi: \C^3 \rightarrow \PP^3, \quad (x,y,z)\mapsto [x^2+1 \,:\, y^2+1 \,:\, xy + z^2  \,:\, 1 ].  \]
    We compute a suitable weight vector $\omega$ such that $s\cdot \{x^2+1,\, y^2+1,\, xy + z^2,\, 1\}$ is SAGBI basis with respect to $(-1,\, \omega)$ using the function \texttt{get\char`_weight}.
\begin{minted}{julia}
using Pkg; Pkg.add("SagbiHomotopy")
using SagbiHomotopy, HomotopyContinuation
@var x,y,z
sagbi = [x^2+1, y^2+1, x*y + z^2, 1]
w = get_weight(sagbi)   #(-2,-2,-3)
\end{minted}
    The functions \texttt{degree\char`_map} and \texttt{degree\char`_monomial\char`_map} compute, respectively, the degree of the parameterization $\phi$ and the degree of the monomial dominant map into $\PP^3$:
    \[  \phi_0: \C^3 \rightarrow \PP^3, \quad (x,y,z)\mapsto [x^2 \,:\, y^2 \,:\, z^2  \,:\, 1 ]. \]
\begin{minted}{julia}
degree_map(sagbi) #8
degree_monomial_map(sagbi, w) #8
\end{minted}
    To define a zero-dimensional linear system on $\PP^3$, we introduce a set of new \texttt{HomotopyContinuation} variables, one for each polynomial in the parameterization, and define \texttt{lin\char`_sys} as $3$ linear equations in the new  variables. We apply our main function \texttt{sagbi\char`_homotopy} to solve the equations on $\PP^3$, providing the computed weight vector $\omega$ as an optional input. 
    The maps $\phi$ and $\phi_0$ both have degree~$8$, which \texttt{sagbi\char`_homotopy} function verifies by default, thus the condition in \Cref{prop: number sols Sagbi homotopy} is satisfied and the SAGBI homotopy computes all the $8\cdot 1$ expected solutions of the system. In particular, in this case the base locus is obviously empty as the polynomial $1$ is in the parameterization. 
\begin{minted}{julia}
@var p[1:4]
lin_sys = [ rand(-100:100,3,4)*p ]
sagbi_homotopy(lin_sys, sagbi; weight = w, degreeCheck = true)  #8 solutions
\end{minted}
\end{examplex}
\begin{examplex}[
$n = 3,\,  m=2, \, k_1=3,\, k_2=3$]  \label{ex: showing functions}
    We consider a semimixed system of type $(1,2)$ of linear equations on the multiprojective variety defined in  \Cref{ex: homogenization sagbi homotopy system}. The  parameterization map $\phi$ and the corresponding monomial map $\phi_0$ have both degree one. 
\begin{minted}{julia}
@var x, y, z
sagbi = [[x, y, (x^2 + y^2), 1], [y, z, (x^2 + y^2), (x^3 + z^3)]]
w = get_weight(sagbi)   #(-1,-2,-3)
degree_map(sagbi) #1
degree_monomial_map(sagbi, w) #1
\end{minted}
    We keep track of the different sets $\mathcal{B}_1$ and $\mathcal{B}_2$ by defining \texttt{sagbi} as a vector of vectors. 
    To get a zero-dimensional system, we define $3$ linear equations subdivided into two sets. The first  consists of two equations in the same set of variables \texttt{p[1:4]}, one for each polynomial in $\mathcal{B}_1$, and the second set consists of a single equation in other variables \texttt{q[1:4]}, corresponding to the four polynomials in $\mathcal{B}_2$. We use random integers in the interval $[-100,100]$ to define the equations.
\begin{minted}{julia}
@var p[1:4]
@var q[1:4]
lin_sys = [ rand(-100:100,2,4)*p , rand(-100:100,1,4)*q]
sagbi_homotopy(lin_sys, sagbi; weight = w)  #6 solutions
\end{minted}
\end{examplex}
\begin{examplex}[$n=2,\, m=1, \,k_1=3$] \label{ex: base locus code}
    We return to the SAGBI basis from \Cref{ex: base locus}, which has a non-empty base locus and smaller degree of the induced monomial map. The function \texttt{sagbi\char`_homotopy} still computes part of the solutions but returns an error message saying that there are missing solutions. In this case we obtain only two of the four expected  solutions.
\begin{minted}{julia}
@var x,y
sagbi = [[x*(x^2 + y^2 - 2*x), x*(5-4*y), y*(x^2 + y^2 - 2*x), y*(5-4*y)]]
w = get_weight(sagbi) # [-2,-1]
degree_map(sagbi) # 2
degree_monomial_map(sagbi,w) #1

@var p[1:4]
lin_sys = [ rand(ComplexF64,2,4) *p ] 
@time res, sols = sagbi_homotopy(lin_sys, sagbi; weight = w) 
##### Output:
Error: degree of monomial parameterisation drops from 2 to 1.
SAGBI homotopy will not find all the solutions.
SAGBI homotopy successfully completed with 2 solutions.
\end{minted}
    To compute the remaining solutions from the base locus we set the optional input \texttt{getBaseLocus} to true in the main function. Note that $(0,0)$ will still be missing, as we compute the base locus inside the torus with the default polyhedral homotopy method. 
\begin{minted}{julia}
res, sols = sagbi_homotopy(lin_sys, sagbi; weight = w, getBaseLocus = true)
##### Output:
Error: degree of monomial parameterisation drops from 2 to 1. 
SAGBI homotopy will not find all the solutions.
SAGBI homotopy successfully completed with 4 solutions.
\end{minted}
\end{examplex}
\begin{examplex}
  We conclude this section with an example of a \textit{special} polynomial square system. The set $\mathcal{B} = \mathcal{B}_1$ is a birational parameterization of the Grassmannian $\Gr(3,6)$, as described in detail in \Cref{sec:5}. Our system consists of 9 specially chosen linear equations in $\mathcal{B}_1$ elements. More precisely, the coefficients of these equations are stored in a sparse $9 \times 20$ block-diagonal matrix:
\[A= \tiny
\left[
\begin{array}{cccccccccccccccccccc}
4 & 1 & 3 & 0 & 8 & -5 & 4 & 0 & 0 & 0 & 0 & 0 & 0 & 0 & 0 & 0 & 0 & 0 & 0 & 0 \\
-3 & 5 & 5 & 0 & -8 & 9 & -2 & 0 & 0 & 0 & 0 & 0 & 0 & 0 & 0 & 0 & 0 & 0 & 0 & 0 \\
8 & 6 & -8 & -10 & -4 & 10 & 9 & 0 & 0 & 0 & 0 & 0 & 0 & 0 & 0 & 0 & 0 & 0 & 0 & 0 \\
0 & 0 & 0 & 0 & 0 & 0 & 0 & 0 & -10 & -2 & 4 & 7 & 7 & 0 & 0 & 0 & 0 & 0 & 0 & 0 \\
0 & 0 & 0 & 0 & 0 & 0 & 0 & 8 & -7 & 6 & 10 & 2 & 10 & 0 & 0 & 0 & 0 & 0 & 0 & 0 \\
0 & 0 & 0 & 0 & 0 & 0 & 0 & -1 & 3 & -9 & 8 & 2 & 2 & 0 & 0 & 0 & 0 & 0 & 0 & 0 \\
0 & 0 & 0 & 0 & 0 & 0 & 0 & 0 & 0 & 0 & 0 & 0 & 0 & 5 & 8 & 5 & -9 & -10 & 8 & 6 \\
0 & 0 & 0 & 0 & 0 & 0 & 0 & 0 & 0 & 0 & 0 & 0 & 0 & -7 & 5 & 0 & 4 & 9 & -2 & -4 \\
0 & 0 & 0 & 0 & 0 & 0 & 0 & 0 & 0 & 0 & 0 & 0 & 0 & 4 & -2 & -2 & 1 & -2 & -9 & 1 \\
\end{array}
\right] 
\]
This system has 12 non-singular solutions. We compute these solutions by tracking $42 = \deg \Gr(3,6)$ paths with an additional feature of varying linear equations $\texttt{varyLinearPart = true}$ with the (randomized) straight line homotopy
\[\zeta\cdot tA\mathbf{z} + (1-t)\widehat{A}\mathbf{z}, \quad \zeta \in \C^{\times},\]
where $\widehat{A}$ is a $9\times 20$ matrix with random complex entries. All the computations for this example are available on the 
MathRepo~\cite{mathrepo} page \eqref{eq: mathrepo}.

In particular, with $\texttt{varyLinearPart = true}$, user can apply SAGBI homotopy to systems parameterized by a set $\mathcal{B}$, such that $\mathbf{s}\mathcal{B}$ doesn't form a SAGBI basis for any term order, but it can be completed to a finite one in the same graded piece of the algebra generated by $\mathbf{s}\mathcal{B}$.
\end{examplex}

\section{Applications}\label{sec:5}

In this section we illustrate big families of horizontally parameterized  systems naturally arising from different areas of mathematics and science whose structure naturally arises as described in \eqref{linearcombination}.
\subsection{Linear equations on the Grassmannian}\label{sec: Grass}

The complex \emph{Grassmannian} $\Gr_\C(k,m)$ parameterizes $k$-dimensional subspaces of $\C^m$.
It is a rational variety embedded into $\PP^{\binom{m}{k}-1}$ via the \emph{Plücker embedding} $ \Gr_\C(k,m)\hookrightarrow \PP^{\binom{m}{k}-1}$, defined using the \emph{Plücker coordinates} $p_I$, $I\in\binom{m}{[k]}$, that are the maximal minors of a general $k\times m$ matrix, see for example \cite[Ch.~5]{Michalek2021Invit-53971}. The coordinate ring of $\Gr_{\C}(k,m)$ is generated by these minors.

We can consider a point in the affine chart of the Grassmannian given by $\Gr_\C(k,m)\cap\{p_{[k]}=1\}$ as a matrix $H$ with a $k\times k$ identity block on the left and $n$ variables on the remaining part, where $n=k(m-k)$ is the dimension of $\Gr_\C(k,m)$. We take the parameterizing set $\mathcal{B}$ to be the set of all maximal minors of $H$. The set $s\mathcal{B}$ forms a SAGBI basis with respect to the \emph{PBW term order}, represented by the following weight vector. See, for instance, \cite[Theorem 3.7]{FeiPopMakh}.
\begin{equation} \label{eq: w grassmannian} \small 
\omega = \begin{pmatrix}
    0 & \cdots & 0 \\
    m-k & \cdots & 1 \\
    2(m-k) & \cdots & 2 \\
    \vdots & & \vdots \\
    (m-k-2)(m-k) & \cdots & (m-k)-2
\end{pmatrix}
\end{equation} 
This gives us a toric degeneration of the Grassmannian to the variety whose coordinate ring is $\mathrm{in}_{\omega}(\C[\Gr_\C(k,m)])$. Moreover, the map given by $\mathcal{B}$ and the induced monomial map $H\mapsto (\mathrm{in}_\omega(p_I))_I$ parameterizing the toric variety have both degree $1$, thus we can apply \Cref{alg: two-step sagbi algorithm} to solve $n=k(m-k)$ linear equations on the Grassmannian via SAGBI homotopy. This is a computational problem arising in many different applications. In \cite{BETTI2025102340} the authors present different Schubert problems, that are given by linear conditions on a Grassmannian and describe incidence conditions on hyperplanes. Efficiently slicing the Grassmannian $\Gr(2,6)$ with $8$-hyperplanes is also one of the main computational tasks to solve the Mukai lifting problem for self-dual points in $\mathbb{P}^6$. In this case, the solutions represent a configuration of $14$ self-dual points in $\mathbb{P}^6$ embedded in $\mathbb{P}^{14}$, see \cite{betti2024mukailiftingselfdualpoints}.
We show this in the following snippet of code.
\begin{minted}{julia}
(x, sagbi, w) = get_sagbi_grassmannian(2,6)
A = randn(ComplexF64,8,15) 
@var z[1:z] #homotopy continuation variables
F = A*z #linear equations on Gr(2,6)
@time res, sols = sagbi_homotopy(F, sagbi, weight = w);
Tracking 14 paths... 100%|| Time: 0:00:02
  # paths tracked:                  14
  # non-singular solutions (real):  14 (0)
  # singular endpoints (real):      0 (0)
  # total solutions (real):         14 (0)
SAGBI homotopy successfully completed with 14 solutions. 
\end{minted}

In \Cref{tab: computational time grassmannian} we report computational times for solving linear equations on different Grassmannians in three different ways . We use our \texttt{sagbi\char`_homotopy} function twice, first giving as input the weight vector $\omega$ as in \eqref{eq: w grassmannian} and running the SAGBI detection algorithm the second time. The third column reports the computational time for the standard \texttt{HomotopyContinuation.solve} function, that is significantly larger when the number of solutions is high. We observe that running the SAGBI detection algorithm increases the computational time as expected, in particular for $\Gr(3,7)$ the \texttt{get\char`_weight} function is very slow. If the computation did not complete within a few hours, we omit the corresponding result in the table. 
\begin{table}[h]
\centering
\small
\begin{tabular}{c|ccccc}
$\Gr$ & dim & degree & $t_\omega$ & $t_{\text{SAGBI detection}}$ & $t_{\text{solve}}$\\ \hline
$(2,4)$ & 4  & 2     &  $1.37s$    & $1.53s$    & $1.82s$ \\
$(2,5)$ & 6  & 5     &  $5.01s$    & $6.36s$    & $3.36s$ \\
$(2,6)$ & 8  & 14     &  $5.89s$    & $7.73s$    & $4.06s$ \\ 
$(3,6)$ & 9 & 42     &  $8.42s$    & $9.81s$    & $9.88s$  \\ 
$(2,7)$ & 10 & 42     &  $8.27s$    & $10.37s$    & $7.06s$  \\ 
$(3,7)$ & 12 & 462    &  $22.00s$    & $1177s$    & $746s$ \\ 
$(2,8)$ & 12 & 132    &  $12.96s$    & $18.34s$    & $19.60s$ \\
$(3,8)$ & 15 & 6006   &  $315.7s$    & $\times$    & $\times$ \\
$(4,8)$ & 16 & 24024  &  $1775s$    & $\times$    & $\times$ \\
$(2,9)$ & 14 & 429    &  $24.39s$    & $\times$    & $\times$ \\
\end{tabular}
    \caption{Computational time for linear equations on Grassmannians.}
    \label{tab: computational time grassmannian}
\end{table}
\subsection{Stationary nonlinear dynamics}
Nonlinear differential equations play a central role in modeling complex physical phenomena such as turbulence, chaos, and phase transitions, appearing in all fields from fluid dynamics to quantum mechanics~\cite{Strogatz1994, Falkovich2011,griffin1996bose,Soriente21}. We focus on differential equations of the form
\begin{equation*}
\ddot{X_i}+\omega_{0}^{2}\left(1-\lambda\cos\left(2t\omega\right)\right)X_i+ \gamma\dot{X_i}+ F_{\text{nl}}(X_i) + \sum_{l\neq i}J_{i,l}X_l=0,
\end{equation*}
which describe parametrically driven \emph{coupled nonlinear resonators} with damping and nonlinear restoring force $F_{\text{nl}}(X)=\alpha_1 X^3 + \ldots + \alpha_{n-1} X^{2n-1}$ assumed to be a polynomial of odd degree.

The \emph{Harmonic balance method}~\cite{Krack_2019,hb_julia_code} is a technique for approximating steady-state solutions of nonlinear differential equations by representing the solution as a truncated Fourier series. It reduces the problem to solving a system of nonlinear algebraic equations whose real solutions correspond to physically meaningful steady states. For instance, for~$N$ coupled resonators with nonlinearity of degree $(2n-1)$ and~$M=1$~leading frequency this leads to~$2N$ algebraic equations in~$2N$ variables~$\bfu = (u_1,\ldots,u_N)$ and $\bfv= (v_1,\ldots,v_N)$, which come in~$N$ pairs: 
\begin{equation}\label{equations_N}
\begin{aligned}
& \begin{bmatrix}
p_1(\bfu,\bfv), &q_1(\bfu,\bfv),&\hdots ,&p_N(\bfu,\bfv), &q_N(\bfu,\bfv)    
\end{bmatrix}=0, \text{ where } \\
p_{i}(\bfu,\bfv) &= a_{0,i} + a_{1,i} u_i + a_{2,i} v_i + \ldots + a_{n+1,i}u_i(u_i^2 + v_i^2)^{n-1} + \frac{1}{2}\sum_{j\neq i}J_{j,i}u_j,\\
q_{i}(\bfu,\bfv)&= b_{0,i} + b_{1,i} u_i + b_{2,i} v_i + \ldots + b_{n+1,i}v_i(u_i^2 + v_i^2)^{n-1} +  \frac{1}{2}\sum_{j\neq i}J_{j,i}v_j\\
\end{aligned}
\end{equation}
with $a_{2,i} = -b_{1,i}, \, a_{k,i} = b_{k,i}$ for~$i=1,\dots,N$ and~$k = 3,\dots,n+1$. For more details  see~\cite{oscillators}.
\begin{examplex}
    Consider a nonlinear resonator with $M=5$ leading frequencies. The solution set of the following square system $\mathcal{F} = \{p_1 = q_1 = \cdots = p_5 = q_5 = 0\}$ of 10 polynomial equations as in~\eqref{equations_N}
is a set of approximate stationary states for the resonator. We solve the system with random \texttt{Float64} coefficients $c_{i,j}$ and $b_{i,j}$ applying our \texttt{sagbi\char`_homotopy} function. The output of \texttt{get\char`_weight} is $\omega=(2,2,2,2,2,1,1,1,1,1)$. The system has $5^5 = 3125$ non-singular complex solutions, one of which is real. The SAGBI homotopy computes them by tracking exactly $3125$ paths in 6 seconds. 
In contrast, the default polyhedral homotopy  tracks 59049 paths and takes $07:15$ minutes. 
\end{examplex}
\begin{remark}
    For $N$ nonlinear resonators of degree $2n-1$ with $M$ leading frequencies, the system has $2NM$ equations in $2NM$ variables of degree $2n-1$. Therefore, the complexity of solving these problems grows extremely fast, making polyhedral homotopy extremely inefficient.  In \cite{oscillators} the SAGBI property for the systems from this family  was showed for all $n, N, M$. 
\end{remark}
In \Cref{tab: computational time resonators} we compare the computational time of our \texttt{sagbi\char`_homotopy} function with the standard \texttt{solve}, testing also the SAGBI detection algorithm.
\begin{table}[h]
\centering
\small
\begin{tabular}{cc|cccc}
$N$ & $M$ & degree & $t_\omega$ & $t_{\text{SAGBI detection}}$ & $t_{\text{solve}}$\\ \hline
$1$ & $1$  & 5       &  $4.78s$     & $4.86s$     & $2.21s$ \\
$1$ & $2$  & 25      &  $3.94$      & $4.21s$     & $2.35s$ \\
$1$ & $3$  & 125     &  $4.65s$     & $278s$      & $3.16s$ \\ 
$1$ & $4$  & 615     &  $6.04s$     & $\times$    & $15s$  \\ 
$1$ & $5$  & 3125    &  $14.17s$    & $\times$    & $196s$  \\ 
$2$ & $1$  & 25      &  $3.86s$     & $2.49$      & $1.902s$ \\ 
$2$ & $2$  & 616     &  $6.34s$     & $\times$    & $11.74s$ \\
$2$ & $3$  & 14887   &  $2795s$     & $\times$    & $2955s$\\
\end{tabular} 
    \caption{Computational time $N$ nonlinear resonators with $M$ leading frequencies.}
    \label{tab: computational time resonators}
\end{table}

\subsection{Coupled cluster equations in quantum chemistry}
Electronic structure theory is a fundamental quantum mechanical framework for describing electron interactions in molecules and crystals. 
The central problem is the solution of the electronic Schrödinger equation
\begin{equation}\label{eq: Schrodinger}
    \mathcal{H} \,\Psi(\mathbf{x}_1, \mathbf{x}_2, \ldots, \mathbf{x}_d) = \lambda\, \Psi(\mathbf{x}_1, \mathbf{x}_2, \ldots, \mathbf{x}_d), 
\end{equation} which leads to a high-dimensional eigenvalue problem with $3d$ degrees of freedom for $d$ electrons in $n$ spin orbitals~\cite{Schneider2009}. The unknown in~\eqref{eq: Schrodinger} is the \emph{wave function} $\Psi(\mathbf{x}_1, \mathbf{x}_2, \ldots, \mathbf{x}_d)$.
The arguments in this function are pairs $\mathbf{x}_i = (\mathbf{r}_i, s_i)$, where $\mathbf{r}_i = (r_i^{(1)}, r_i^{(2)}, r_i^{(3)})$ are points in $\mathbb{R}^3$ representing the positions of $d$ electrons, and $s_i \in \{\pm 1/2\}$ describes the electronic spin. As the size of the system increases, the computational complexity grows exponentially, requiring sophisticated numerical methods. Among the most accurate and widely used of these is the coupled cluster theory~\cite{Helgaker2000}.
From the  equation \eqref{eq: Schrodinger}, we  derive a nonlinear eigenvalue problem, as explained in detail in \cite{FSS}:
\begin{equation}\label{eq: eigenvalueproblem}
   H \psi = \lambda \psi.
\end{equation}
Vectors $\psi$ are called \emph{quantum states}.
Their entries are, in fact, Plücker coordinates.
After truncating~\eqref{eq: eigenvalueproblem} to a square system, this becomes a horizontally parameterized system on a certain variety. We refer to the polynomial equations in \eqref{eq: eigenvalueproblem} as \emph{coupled cluster equations}. For general real symmetric matrix $H$, called a \emph{Hamiltonian}, the number of complex solutions is constant and called  \emph{CC (coupled cluster) degree}. For a \emph{coupled cluster single} model (\cite{FO, faulstich2023homotopy}), it was showed in \cite{meBerndSvala} that the system \eqref{eq: eigenvalueproblem} is a linear section of the affine cone over the graph of a birational parameterization of the Grassmannian (the same as in \Cref{sec: Grass}). In \Cref{tab: computational time cc equations} we collect the computational time  for solving CC equations for different $d$ and $n$.
\begin{table}[h]
\centering
\small
\begin{tabular}{c|ccccc}
$(d,n)$ & dim & CCdegree & $t_\omega$ & $t_{\text{SAGBI detection}}$ & $t_{\text{solve}}$\\ \hline
$(2,4)$ & 4  & 9      &  $1.47s$    & $1.59s$    & $1.95s$ \\
$(2,5)$ & 6  & 27      &  $5.08$    & $5.28s$    & $2.90s$ \\
$(2,6)$ & 8  & 83     &  $6.23s$    & $7.40s$    & $4.38s$ \\ 
$(3,6)$ & 9 & 250      &  $9.06s$    & $13.76s$    & $21.14s$  \\ 
$(2,7)$ & 10 & 263    &  $9.38s$    & $14.13s$    & $10.48s$  \\ 
$(3,7)$ & 12 & 2883    &  $54.10s$    & $2111s$    & $11203s$ \\ 
$(2,8)$ & 12 & 857    &  $19.60s$    & $55.50s$    & $61.17s$ \\
$(3,8)$ & 15 & 38607   &  $1818s$    & $\times$    & $\times$ \\
\end{tabular}
    \caption{Computational time for coupled cluster equations on Grassmannians.}
    \label{tab: computational time cc equations}
\end{table}

\appendix
\section{Appendix} \label{appendix}
\begin{proposition} \label{prop: deformalgebras}
    Let $\{b_0,\dots,b_k\} \subset \C[x_1,\dots,x_n]$ form a SAGBI basis with respect to~$\omega$ for the subalgebra $S=\C[b_0,\dots,b_k]$ and take $\nu = \mathcal{A}^{T}\omega$, where $\mathcal A$ is a $(k+1)\times n$ matrix with columns given by the exponents of $\omega$-leading monomials of $\{b_0,\dots,b_k\}$.  Then for every $t \in \C^1$ the subalgebra $S_t^{\omega}$ is isomorphic to  $\C[\mathbf{z},t]/I_t^{\nu}$, where $I \subseteq \C[z_0,\dots,z_k]$ is the ideal of algebraic relations of $\{b_0,\dots,b_k\}$.
\end{proposition}
To prove \Cref{prop: deformalgebras} we first consider two technical lemmas.
\begin{lemma}\label{lem: deformalgebras1}
    Products of polynomials commute with taking homogenizations.
That is,
\[
(p h)^{\omega}_t = p_t^{\omega} h_t^{\omega}
\]
\end{lemma}
\begin{proof}
Take two  polynomials in $\C[\bfx] = \C[x_1, \ldots, x_n]$:
    \[p = \sum_{\langle \omega,\alpha \rangle \le m} c_{\alpha} \bfx^\alpha, \quad h = \sum_{\langle \omega,\alpha \rangle \le \ell} d_\alpha \bfx^\alpha.
\]
There is a natural grading on $\C[\bfx]$ where the weights of $(x_1, \ldots, x_n)$ are given by $\omega=(\omega_1, \ldots, \omega_n)$. Each polynomial $p \in \C[\bfx]$ decomposes into graded pieces as
\[
p = \sum p_i, 
\]
where $p_i \in \C[\bfx]_i$ is the set of polynomials of $\omega$-degree $i$,
that is, $\omega \cdot \alpha = i$ when $\bfx^\alpha$ appears in~$p_i$.
Since this is a grading, we have $p_i p_j \in \C[\bfx]_{i+j}$.
Note that our homogenization works as follows: for polynomials $p$ and $h$ we have decompositions $p = \sum_{i=0}^m p_i$, $h = \sum_{j=0}^\ell h_j$,
and
\[
p_t^{\omega} = \sum_{i=0}^m p_i t^{m-i}, \quad h_t^{\omega} = \sum_{j=0}^\ell h_j t^{\ell-j}.\]
Obviously, every homogeneous component of $p h$ arises as sum of products of $p_i$ and $h_j$:
\[
p h = \sum_{i,j} (p_i h_j) = \sum_{i=0}^m \sum_{j=0}^\ell p_i h_j,
\]
where each $p_i h_j$ is homogeneous of degree $i+j$.
To homogenize, we multiply each $p_i h_j$ by $t^{m+\ell-(i+j)}$:
\[
(p h)_t^\omega = \sum (p_i h_j) t^{m+\ell-(i+j)}.
\]
We obtain exactly the same sum by multiplying $p_t^\omega$ by $h_t^\omega$.
\end{proof}
\begin{remark}
    \Cref{lem: deformalgebras1} can be clearly generalized for products of an arbitrary number of polynomials.
\end{remark}
\begin{lemma}\label{lem: deformalgebras2}
    Let $p \in \C[\mathbf z] = \C[z_0, \ldots, z_k]$ be a $\nu$-homogeneous polynomial, $\nu =\mathcal{A}^{T}\omega$, that is,
\[
    p = \sum_{\langle \nu, \alpha \rangle = i} c_\alpha {\bf z}^\alpha.
\]
Then we have  $\left[ p(b_0, \ldots, b_k) \right]^\omega_t = p(b_{0,t}^\omega, \ldots, b_{k,t}^\omega)$.
\end{lemma}
\begin{proof}
    We have $(b_{0,t}^{\omega})^{\alpha_0} \cdots (b_{k,t}^{\omega})^{\alpha_k}
    = (b_0^{\alpha_0} \cdots b_k^{\alpha_k})_t^\omega$
by \Cref{lem: deformalgebras1}.
We consider the $\omega$-decomposition 
\begin{equation}\label{eq: monomialinfDec}
     b_0^{\alpha^{(j)}_0} \cdots b_k^{\alpha^{(j)}_k} = \sum_{s \leq i} r_s^{(j)}.
\end{equation}
The $\omega$-order of $\operatorname{in}_\omega(b_0^{\alpha_0} \cdots b_k^{\alpha_k})$ is $i = \langle \nu, \alpha \rangle$, which  does not depend on $\alpha$ from $p$, since $p$ is $\nu$-homogeneous.
Indeed, writing $b_i = {\bf x}^{\beta_i} + \sum c_\alpha {\bf x}^\alpha$, where ${\rm in}_{\omega}(b_i) = \bfx^{\beta_i}$, we conclude that 
\[
    \operatorname{in}_\omega(b_0^{\alpha_0} \cdots b_k^{\alpha_k}) = \operatorname{in}_\omega(b_0)^{\alpha_0} \cdots \operatorname{in}_\omega(b_k)^{\alpha_k}
    = \bfx^{\alpha_0 \beta_0+ \cdots + \alpha_k \beta_k}.
\]
The $\omega$-order of this monomial is  $\langle \omega, \alpha_0 \beta_0 + \cdots + \alpha_k \beta_k \rangle = \langle \mathcal{A}^{T} \omega, \alpha \rangle = \langle \nu, \alpha \rangle = i$.
So for the corresponding homogenization from \eqref{eq: monomialinfDec} we obtain:
\[
    (b_0^{\alpha_0^{(j)}} \cdots b_k^{\alpha_k^{(j)}})^\omega_t = \sum_{s \leq i} r_s^{(j)} t^{i-s}.
\]
Now, we decompose $p(b_0, \ldots, b_k)$ into $\omega$-homogeneous parts.
\[
    p(b_0, \ldots, b_k) = \sum_j b_0^{\alpha_0^{(j)}} \cdots b_k^{\alpha_k^{(j)}} = \sum_{s \leq i} \left( \sum_j r_s^{(j)} \right).
\]
Its homogenization is then given by
\[
    \left[ p(b_0, \ldots, b_k) \right]^\omega_t = \sum_{s \leq i}  \sum_j r_s^{(j)}  t^{i-s}.
\]
On the other hand, we obtain
\[
    p(b_{0,t}^\omega, \ldots, b_{k,t}^\omega) = \sum_{j}  (b_{0,t}^{\omega})^{\alpha_0^{(j)}} \cdots (b_{k,t}^{\omega})^{\alpha_k^{(j)}} = \sum_j (b_0^{\alpha_0^{(j)}} \cdots b_k^{\alpha_k^{(j)}})^\omega_t = \sum_j \sum_{s \leq i} r_s^{(j)} t^{i-s}. \qedhere \]
\end{proof}
\begin{proof} (of \Cref{prop: deformalgebras})
    To prove \Cref{prop: deformalgebras} we first show that $$g_{t}^{\nu}(b_{0,t}^\omega, \ldots, b_{k,t}^\omega) = [g(b_0, \ldots, b_k)]_{t}^{\omega}$$ for any $g\in\C[\bf z]$.
Again, we decompose $g$ into $\nu$-graded pieces:
\begin{equation}\label{eq: gdecomp}
    g = \sum_{\langle \nu, \alpha \rangle \le q} k_\alpha {\bf z}^\alpha, \quad g = \sum_{i=o}^q g_i,
\end{equation}
where each monomial ${\bf z}^\alpha$ in $g_i$ is such that $\langle \nu, \alpha \rangle = i$.
The homogenization $g_{t}^{\nu}$ is then given by
\[
g^{\nu}_t = \sum_{i=0}^q g_i t^{q - i}.
\]
Meanwhile, we have:
\[
\left[ g(b_0, \ldots, b_k) \right]^\omega_t = \left[\sum_{i=0}^q g_i(b_0, \ldots, b_k)\right]^\omega_t = \sum_{i=0}^q \left[g_i(b_0, \ldots, b_k)\right]^\omega_t t^{q - i}.
\]
The last equality comes from the fact that, when homogenizing, we work with every monomial separately but then multiply the whole polynomial by $t^N$, where $N = \max \langle \nu, \alpha \rangle$ for this polynomial.
\Cref{lem: deformalgebras2} shows that such maximum for each $g_i$ is $\langle \nu, \alpha \rangle = \langle \mathcal{A}^T \omega, \alpha \rangle = i$.
That is,
\[
\left[ g_i(b_0, \ldots, b_k) \right]^\omega_t = t^{i}  \sum_\alpha (b_0^{\alpha_0} \cdots b_k^{\alpha_k})^\omega_t.
\]
However, 
\[
\left[ \sum_{i=0}^q g_i(b_0, \dots, b_k) \right]^\omega_t 
= t^q  \sum_{i=0}^q \sum_\alpha (b_0^{\alpha_0} \cdots b_k^{\alpha_k})^\omega_t,
\]
since $q = \langle \nu, \alpha \rangle = \langle \mathcal{A}^T \omega, \alpha \rangle$ is maximum on all terms of $g$.
On the other hand,
\[g^{{\nu}}_t(b_{0,t}^\omega, \ldots, b_{k,t}^\omega) = \sum_{i=0}^q g_i(b_{0,t}^\omega, \ldots, b_{k,t}^\omega) t^{q - i} = \sum_{i=0}^q \left[ g_i(b_0, \dots, b_k) \right]^\omega_t t^{q - i},\]
where the first equality is obtained from the decomposition \eqref{eq: gdecomp} and the second one is true by \Cref{lem: deformalgebras2}. Thus, we obtain
\[
\left[ g(b_0, \dots, b_k) \right]^\omega_t = g^{\nu}_t(b_{0,t}^\omega, \ldots, b_{k,t}^\omega).
\]
This implies that $g(b_0, \dots, b_k) = 0$ if and only if $g^{\nu}_t(b_{0,t}^\omega, \ldots, b_{k,t}^\omega) = 0$. Hence, $I_{t}^\nu$ is contained in the ideal of the algebraic relations between $\{b_{0,t}^\omega, \ldots, b_{k,t}^\omega\}$. The equality follows from the fact that $\C[\mathbf{z}, t]/I_{t}^\nu$ defines a flat family (cf.~\Cref{prop: flatfactorrings}), so the dimension and primality are preserved.
\end{proof}


\section*{Acknowledgements}

We are grateful to Paul Breiding, Carles Checa and Simon Telen for helpful discussions and their comments on the manuscript. We thank Ben Hollering for suggestions regarding the implementation.

\printbibliography

\bigskip 

\noindent
\small
{\bf Authors' addresses:}

\smallskip

\noindent Barbara Betti, \\
Max Planck Institute for Mathematics in the Sciences, Leipzig, Germany 
\hfill {\tt betti@mis.mpg.de}

\medskip

\noindent Viktoriia Borovik,\\
Max Planck Institute for Mathematics in the Sciences, Leipzig, Germany 
\hfill {\tt borovik@mis.mpg.de}

\end{document}